\documentclass{conm-p-l}
\usepackage{amssymb}

\newcommand{\dup}[2]{\langle#1,#2 \rangle}

\newcommand{\ip}[2]{(#1,#2)}
\newcommand{\Co}{\mathrm{Co}} 
\newcommand{\cdual}{{\ast}}
\newcommand{\cone}{{\Omega}}
\newcommand{\besov}{\dot{B}}
\newcommand{\supp}{\mathrm{supp}}
\newcommand{\pdif}{\partial}
\newcommand{\grG}{\mathcal{G}}
\newcommand{\grH}{\mathcal{H}}
\newcommand{\grK}{\mathcal{K}}
\newcommand{\grA}{\mathcal{A}}
\newcommand{\grN}{\mathcal{N}}

\newtheorem{theorem}{Theorem}[section]
\newtheorem{lemma}[theorem]{Lemma}
\newtheorem{corollary}[theorem]{Corollary}

\theoremstyle{definition}

\newtheorem{remark}[theorem]{Remark}

\begin{document}

\title{Atomic decompositions of Besov spaces related to symmetric cones} 

\author{Jens Gerlach Christensen} 
\address{ Department of Mathematics, Tufts University} 
\email{jens.christensen@tufts.edu}
\urladdr{http://jens.math.tufts.edu}

\begin{abstract}
In this paper we extend the atomic decompositions 
previously obtained for Besov spaces related to the forward light cone
to general symmetric cones.
We do so via wavelet theory adapted 
to the cone. The wavelet transforms sets up an isomorphism between
the Besov spaces and certain reproducing kernel function spaces
on the group, and sampling of the transformed data will provide
the atomic decompositions and frames for the Besov spaces.
\end{abstract}

\maketitle

\section{Introduction}
\label{sec:intro}
Besov spaces related to symmetric cones were introduced by
Bekoll\'e, Bonami, Garrigos and Ricci in a
series of papers \cite{Bekolle2000, Bekolle2001} and \cite{Bekolle2004}.
The purpose was to use Fourier-Laplace extensions
for the Besov spaces in order to investigate the continuity of 
Bergman projections and boundary values for Bergman spaces 
on tube type domains.

Classical homogeneous Besov spaces were introduced via 
local differences and modulus of continuity.
Through work of Peetre \cite{Peetre76}, Triebel \cite{Triebel1988b}
and Feichtinger and Gr\"ochenig \cite{Feichtinger1989a} these spaces
were given a characterization via wavelet theory.
The theory of Feichtinger and Gr\"ochenig 
\cite{Feichtinger1989a,Grochenig1991}
further provided atomic decompositions
and frames for the homogeneous Besov spaces.

In the papers \cite{Christensen2011} and \cite{Christensen2012}
we gave a wavelet characterization and several atomic decompositions
for the Besov spaces related to the special case of
the forward light cone. In this paper we will show that the
machinery carries over to  Besov spaces related to any symmetric cone.
Our approach contains some representation theoretic simplifications
compared with the work of Feichtinger and Gr\"ochenig, and we
in particular exploit smooth representations of Lie groups.
The results presented here 
are also interesting in the context of recent results by F\"uhr \cite{Fuhr2013}
dealing with coorbits for wavelets with general dilation groups.

\section{Wavelets, sampling and atomic decompositions}
In this section we use representation theory to
set up a correspondance between
a Banach space of distributions and a reproducing kernel Banach space on
a group. For details we refer to 
\cite{Christensen2009,Christensen2011,Christensen2012}
which generalizes work in \cite{Feichtinger1989a}.

\subsection{Wavelets and coorbit theory}
\label{sec:coorbit}

Let $S$ be a Fr\'echet space and let $S^\cdual$ be the 
conjugate linear dual equipped with the weak* topology (any reference 
to weak convergence in $S^*$ will always refer to the weak* topology). 
We assume that $S$ is
continuously embedded and 
weakly dense in $S^\cdual$. The conjugate
dual pairing of elements $\phi\in S$ and $f\in S^\cdual$ will be denoted
by $\dup{f}{\phi}$. Let $\grG$ be a locally compact 
group with a fixed left Haar measure
$dg$, and assume that $(\pi,S)$ is a continuous representation of
$\grG$, i.e. $g\mapsto \pi(g)\phi$
is continuous for all $\phi\in S$. 
A vector $\phi\in S$ is called \emph{cyclic} if 
$\dup{f}{\pi(g)\phi}=0$ for all $g\in \grG$ means that
$f=0$ in $S^*$.
As usual, define the contragradient
representation $(\pi^\cdual,S^\cdual)$ by
\begin{equation*}
  \dup{\pi^\cdual(g)f}{\phi}=\dup{f}{\pi(g^{-1})\phi} \text{\ for $f\in S^\cdual.$}
\end{equation*}
Then $\pi^*$ is a continuous representation of $\grG$ on $S^\cdual$. 
For a fixed vector 
$\psi\in S$ define the linear map $W_\psi:S^*\to C(\grG)$ by
\begin{equation*}
  W_\psi(f)(g) = \dup{f}{\pi(g)\psi} = \dup{\pi^*(g^{-1})f}{\psi}.
\end{equation*}
The map $W_\psi$ is called \emph{the voice transform} or 
\emph{the wavelet transform}.
If $F$ is a function on $\grG$ then define the left translation of $F$
by an element $g\in \grG$ as
\begin{equation*}
  \ell_g F(h) = F(g^{-1}h).
\end{equation*}
A Banach space of functions $Y$ is called left
invariant if $F\in Y$ implies that $\ell_g F\in Y$ for all
$g\in \grG$ and there is a constant 
$C_g$ such that $\| \ell_g F\|_Y\leq C_g \| F\|_Y$
for all $F\in Y$.
In the following we will always assume that 
the space $Y$ of functions on $\grG$ is a left invariant Banach space
for which 
convergence implies convergence (locally) in Haar measure on $\grG$.
Examples of such spaces are $L^p(\grG)$ for $1\leq p \leq \infty$ 
and any space continuously included in an $L^p(\grG)$.

A non-zero cyclic vector $\psi$ 
is called an \emph{analyzing vector} for $S$ if for all $f\in S^*$ 
the following convolution reproducing formula holds
\begin{equation*}
  W_\psi(f)*W_\psi(\psi) = W_\psi(f).
\end{equation*}
Here convolution between two functions $F$ and $G$ on $\grG$ 
is defined by $$F*G(h) = \int F(g)G(g^{-1}h)\,dg. $$

For an analyzing vector $\psi$ define the subspace $Y_\psi$ of $Y$ by 
$$Y_\psi = \{ F\in Y\, |\, F=F*W_\psi(\psi) \},$$ 
and let $$\Co_S^\psi Y = \{ f\in S^*\, |\, W_\psi(f)\in Y  \}$$
equipped with the norm $\| f \| = \| W_\psi(f)\|_Y$.

A priori we do not know if the spaces $Y_\psi$ and $\Co_S^\psi Y$
are trivial, but the following theorem lists conditions that 
ensure they are isometrically isomorphic Banach spaces.
The main requirements are the existence of a reproducing formula
and a duality condition involving $Y$.

\begin{theorem}
  \label{thm:coorbitsduality}
  Let $\pi$ be a representation of a group $\grG$ on a 
  Fr\'echet space $S$ with conjugate dual $S^*$
  and let $Y$ be a left invariant Banach function space on $\grG$.
  Assume $\psi$ is
  an analyzing vector for $S$ and that the mapping
  \begin{equation*}
    Y\times S\ni (F,\phi) \mapsto \int_\grG F(g) \dup{\pi^*(g)\psi}{\phi}\, dg\in\mathbb{C}
  \end{equation*}
  is continuous. Then 
  \begin{enumerate}
  \item $Y_\psi$ is
    a closed reproducing kernel subspace of $Y$ with reproducing kernel
    $K(g,h)=W_\psi(\psi)(g^{-1}h)$.\label{prop1}
  \item The space $\Co_S^\psi Y$
    is a $\pi^\cdual$-invariant Banach space.     \label{prop2}
  \item $W_\psi:\Co_S^\psi Y\to Y$ intertwines $\pi^\cdual$ and left
    translation. \label{prop3}
  \item If left translation is continuous on $Y,$ then
    $\pi^*$ acts continuously on $\Co_S^\psi Y.$\label{prop4}
  \item $\Co_S^\psi Y = \{ \pi^\cdual(F)\psi \mid F\in
    Y_\psi\}$. \label{prop5}
  \item $W_\psi:\Co_S^\psi Y \to Y_\psi$ is an isometric
    isomorphism.\label{prop6}
  \end{enumerate}
\end{theorem}

Note that (\ref{prop5}) states that each member of $\Co_S^\psi Y$
can be written weakly as
\begin{equation*}
  f = \int_\grG W_\psi(f)(g)\pi^*(g)\psi\,dg.
\end{equation*}
In the following section we will explain when this reproducing formula
can be discretized and how coefficents $\{ c_i(f)\}$
can be determined in order to obtain an expression
\begin{equation*}
  f = \sum_i c_i(f) \pi^*(g_i)\psi
\end{equation*}
for any $f\in \Co_S^\psi Y$.

\subsection{Frames and atomic decompositions through sampling on Lie groups}
\label{sec:sampling}
In this section we will decompose the coorbit spaces constructed in 
the previous section. For this we need sequence spaces 
arising from Banach function spaces on $\grG$. 
The decomposition of coorbit spaces is aided by smooth representations
of Lie groups.

We assume that $\grG$ is a Lie group with Lie algebra denoted $\mathfrak{g}$.
A vector $\psi\in S$ is called $\pi$-weakly differentiable in the direction
$X\in\mathfrak{g}$ if there is a vector denoted $\pi(X)\psi \in S$
such that for all $f\in S^*$ 
\begin{equation*}
  \dup{f}{\pi(X)\psi}=
  \frac{d}{dt}\Big|_{t=0}\dup{f}{\pi(e^{tX})\psi}.
\end{equation*}
Fix a basis $\{ X_i\}_{i=1}^{\dim \grG}$ for $\mathfrak{g}$, then
for a multi-index $\alpha$ we define
$\pi(D^\alpha)\psi$ (when it makes sense) by
\begin{equation*}
  \dup{f}{\pi(D^\alpha)\psi}=
  \dup{f}{\pi(X_{\alpha(k)}) \pi(X_{\alpha(k-1)}) \cdots \pi(X_{\alpha(1)})\psi}.
\end{equation*}

A vector $f\in S^*$ is called $\pi^*$-weakly differentiable 
in the direction
$X\in\mathfrak{g}$ if there is a
vector denoted $\pi^*(X)f \in S^*$ 
such that for all $\phi\in S$ 
\begin{equation*}
  \dup{\pi^*(X)f}{\psi}=
  \frac{d}{dt}\Big|_{t=0}\dup{\pi^*(e^{tX})f}{\psi}.
\end{equation*}
For a multi-index $\alpha$ define 
$\pi^*(D^\alpha)\psi$ (when it makes sense) by
\begin{equation*}
  \pi^*(D^\alpha)\psi =
  \pi^*(X_{\alpha(k)}) \pi^*(X_{\alpha(k-1)}) \cdots \pi^*(X_{\alpha(1)})\psi
\end{equation*}

Let $U$ be a relatively compact set in $\grG$
and let $I$ be a countable set. A sequence
$\{ g_i\}_{i\in I}\subseteq \grG$ is called $U$-dense if
$\{ g_iU\}$ cover $\grG$, and $V$-separated if for some
relatively compact set $V\subseteq U$ the $g_iV$ are pairwise disjoint.
Finally we say that $\{ g_i\}_{i\in I}\subseteq \grG$
is well-spread if it is $U$-dense and a finite union of
$V$-separated sequences.
For properties of such sequences we refer to \cite{Feichtinger1989a}.
A Banach space $Y$ of measurable functions
is called solid, if $|f|\leq |g|$, $f$ measurable 
and $g\in Y$ imply that $f\in Y$. 
For a $U$-relatively separated sequence of points $\{g_i \}_{i\in I}$ in $\grG$,
and a solid Banach function space $Y$ on $\grG$,
define the space $Y^\#(I)$ of sequences $\{\lambda_i\}_{i\in I}$ for 
which 
\begin{equation*}
  \| \{ \lambda_i\}\|_{Y^\#} 
  := \left\| \sum_{i\in I} |\lambda_i|1_{g_iU} \right\|_Y <\infty.
\end{equation*}
These sequence spaces were introduced in \cite{Feichtinger1985} (see also 
\cite{Feichtinger1989a}), and we remark that they are independent
on the choice of $U$ (for a fixed well spread sequence).
For a well-spread set $\{g_i\}$ a 
$U$-bounded uniform partition of unity ($U$-BUPU) is a collection
of functions $\psi_i$ on $\grG$ such that $0\leq \psi_i \leq 1_{g_iU}$
and $\sum_i \psi_i =1$.

In the sequel we will only investigate 
sequences which are well-spread with respect to 
compact neighbourhoods of the type
\begin{equation*}
  U_\epsilon = \{ e^{t_1X_1}\cdots e^{t_nX_n}  
  \mid t_1,\dots,t_n \in [-\epsilon,\epsilon] \},
\end{equation*}
where $\{X_i\}_{i=1}^n$ is the fixed basis for $\mathfrak{g}$.
\begin{theorem}
  \label{thm:discretization}
  Let $Y$ be a solid and left and right invariant 
  Banach function space for which right translations are continuous.
  Assume there is a cyclic vector $\psi\in S$ 
  satisfying the properties of Theorem~\ref{thm:coorbitsduality}.
  and that $\psi$ is both $\pi$-weakly and $\pi^*$-weakly 
  differentiable up to order
  $\mathrm{dim}(\grG)$.
  If the mappings
    $Y\ni F\mapsto F*|W_{\pi(D^\alpha)\psi}(\psi)|\in Y$
  are continuous for all $|\alpha|\leq \mathrm{dim}(\grG)$,
  then we can choose $\epsilon$ and positive constants $A_1,A_2$ 
  such that for any $U_\epsilon$-relatively separated set $\{ g_i\}$ 
  \begin{equation*}
    A_1 \| f\|_{\Co_S^\psi Y} 
    \leq \| \{ \dup{f}{\pi(g_i)\psi} \}\|_{Y^\#}
    \leq A_2 \| f\|_{\Co_S^\psi Y}.
  \end{equation*}
  Furthermore, there is an operator $T_1$ such that
  \begin{equation*}
    f =
    W_\psi^{-1 }T_1^{-1} \left( \sum_i W_\psi(f)(g_i)\psi_i*W_\psi(\psi) \right),
  \end{equation*}
  where $\{\psi_i \}$ is any $U_\epsilon$-BUPU for which
  $\supp(\psi_i)\subseteq g_iU_\epsilon$
\end{theorem}

The operator $T_1:Y_\psi\mapsto Y_\psi$ (first introduced in
\cite{Grochenig1991}) is defined by
\begin{equation*}
  T_1 F = \sum_i F(g_i)\psi_i*W_\psi(\psi).
\end{equation*}

\begin{theorem}
  \label{thm:12}
  Let $\psi\in S$ be $\pi^*$-weakly differentiable up to order
  $\mathrm{dim}(\grG)$ satisfying
  the assumptions in Theorem~\ref{thm:coorbitsduality} and 
  let $Y$ be a solid left and right invariant Banach
  function space for which right translation is continuous.
  Assume that $Y\ni F \mapsto F*|W_\psi(\pi^*(D^\alpha) \psi)| \in Y$
  is continuous for $|\alpha|\leq \mathrm{dim}(\grG)$.
  We can choose $\epsilon$ small enough that 
  for any $U_\epsilon$-relatively separated set $\{ g_i\}$
  there is an invertible operator $T_2$ and functionals
  $\lambda_i$ (defined below) such that for any $f \in \Co_S^\psi Y$ 
  \begin{equation*}
    f = \sum_i \lambda_i(T^{-1}_2 W_\psi(f)) \pi(g_i)\psi
  \end{equation*}
  with convergence in $S^*$. The convergence is in $\Co_S^\psi Y$ if
  $C_c(\grG)$ are dense in $Y$.
\end{theorem}
The operator $T_2:Y_\psi\mapsto Y_\psi$ (also introduced in \cite{Grochenig1991})
is defined by
\begin{equation*}
  T_2 F = \sum_i \lambda_i(F)\ell_{g_i}W_\psi(\psi),
\end{equation*}
where $\lambda_i(F) = \int F(g)\psi_i(g)\,dg$.



\section{Besov spaces on symmetric cones}

\subsection{Symmetric cones}
For an introduction to symmetric cones we refer to the book
\cite{Faraut1994}.
Let $V$ be a Euclidean vector space over the real numbers of finite
dimension $n$.
A subset $\cone$ of $V$ is a cone if $\lambda\cone \subseteq \cone$
for all $\lambda>0$. Assume $\cone$ is open and convex, and
define the open dual cone $\cone^*$ by
\begin{equation*}
  \cone^* = \{y\in V \mid \ip{x}{y} > 0\,\, 
  \text{for all non-zero $x\in \overline{\cone}$} \}.
\end{equation*}
The cone $\cone$ is called symmetric if $\cone=\cone^*$ and
the automorphism group
\begin{equation*}
  \grG(\cone) = \{ g\in\mathrm{GL}(V)\mid g\cone=\cone   \}
\end{equation*}
acts transitively on $\cone$. 
In this case the set of adjoints of elements in $\grG(\cone)$
is $\grG(\cone)$ itself, i.e. $\grG(\cone)^*=\grG(\cone)$.
Define the characteristic function of $\cone$ by
\begin{equation*}
  \varphi(x) = \int_{\cone^*} e^{-\ip{x}{y}}\,dy,
\end{equation*}
then 
\begin{equation*}
  \varphi(gx) = |\det(g)|^{-1}\varphi(x).
\end{equation*}
Also,
\begin{equation}
  \label{eq:invintcone}
  f\mapsto \int_\cone f(x)\varphi(x)\,dx
\end{equation}
defines a $\grG(\cone)$-invariant measure on $\cone$.
The connected component $\grG_0(\cone)$ of $\grG(\cone)$
has Iwasawa decomposition
\begin{equation*}
  \grG_0(\cone)= \grK \grA \grN
\end{equation*}
where $\grK = \grG_0(\cone)\cap \mathrm{O}(V)$ is compact, 
$\grA$ is abelian and $\grN$ is nilpotent.
The unique fixed
point in $\cone$ for the mapping $x\mapsto \nabla \log\varphi(x)$
will be denoted $e$, and we note that $\grK$ fixes $e$.
The connected solvable subgroup $\grH = \grA \grN$ of $\grG_0(\cone)$ 
acts simply transitively on $\cone$ and the integral
(\ref{eq:invintcone}) thus also
defines the left-Haar measure on $\grH$. 
Throughout this paper we will identify
functions on $\grH$ and $\cone$ 
by right-$\grK$-invariant functions on $\grG_0(\cone)$.
If $F$ is a right-$\grK$-invariant function on $\grG$
and we denote by $f$ the corresponding function
on the cone $\cone$, then
\begin{equation*}
  F \mapsto \int_\grH F(h)\,dh 
  := \int_\cone f(x)\varphi(x)\,dx
\end{equation*}
gives an integral formula for the left-Haar measure on $\grH$
which we will denote by $dh$ or $\mu_\grH$. 
\begin{lemma}
  \label{lem:newhaar}
  If $F$ is an $\mu_\grH$-integrable 
  right-$\grK$-invariant function 
  on $\grG_0(\cone)$, then 
  there is a constant $C$ such that
  \begin{equation*}
    \int F(h)\,dh = C \int F((h^*)^{-1}) \,dh.
  \end{equation*}
  Here $h^*$ denotes the adjoint element of $h$ with respect
  to the inner product on $V$.
\end{lemma}

\begin{proof}
  Without loss of generality 
  we will assume that $F$ is compactly supported.
  Note first that the function $h\mapsto F((h^*)^{-1})$ is 
  right-$\grK$-invariant and therefore can be 
  identified with a function on $\cone$.
  Since the measure $\varphi(x)\,dx$ on $\cone$
  is $\grG_0(\cone)$-invariant, the measure
  on $\grH$ is also $\grG_0(\cone)$-invariant.
  For $g\in \grG_0(\cone)$ we have that
  $\ell_gF((h^*)^{-1}) = F(((g^*h)^*)^{-1})$, 
  and therefore 
  the mapping
  \begin{equation*}
    F \mapsto \int_\grH F((h^*)^{-1}) \,dh
  \end{equation*}
  defines a left-invariant measure on $\grH$.
  By uniqueness of Haar measure we conclude that
  \begin{equation*}
    \int_\grH F((h^*)^{-1})\,dh = C \int_\grH F(h)\,dh.
  \end{equation*}
\end{proof}

For $f\in L^1(V)$ the Fourier transform is defined by
\begin{equation*}
  \widehat{f}(w) = \frac{1}{(2\pi)^{n/2}}\int_V f(x)e^{-i\ip{x}{w}}\,dx
  \text{\ for $w\in V,$}
\end{equation*}
and it extends to an unitary operator on $L^2(V)$ in the usual way.
Denote by $\mathcal{S}(V)$ the space of rapidly decreasing
smooth functions with topology induced by the semi-norms
\begin{equation*}
  \| f\|_{k} 
  = \sup_{|\alpha|\leq k} \sup_{x\in V} |\pdif^{\alpha} f (x)|(1+|x|)^k.
\end{equation*}
Here $\alpha$ is a multi-index, $\partial^\alpha$ denotes usual partial
derivatives of functions,  and $k\geq 0$ is an integer.
The convolution 
$$
f*g(x) = \int_V f(y)g(x-y)\,dy
$$
of functions $f,g\in \mathcal{S}(V)$ satisfies
\begin{equation*}
  \widehat{f*g}(w) = \widehat{f}(w)\widehat{g}(w).
\end{equation*}
The space $\mathcal{S}'(V)$ of tempered distributions
is the linear dual of $\mathcal{S}(V)$.
For functions on $V$ define
$\tau_x f(y) = f(y-x)$, $f^\vee(y) = f(-y)$ and
$f^*(y)=\overline{f(-y)}$.
Convolution of $f\in \mathcal{S}'(V)$ and $\phi\in\mathcal{S}(V)$
is defined by 
$$
f*\phi(x) = f(\tau_x\phi^\vee).
$$
The space of rapidly decreasing smooth
functions with Fourier transform vanishing on $\cone$ is
denoted $\mathcal{S}_\cone$. It is a closed subspace of $\mathcal{S}(V)$
and will be equipped with the subspace topology.

The space $V$ can be equipped with a Jordan algebra structure
such that $\overline{\cone}$ is identified with the set of all
squares. This gives rise to the notion of a determinant
$\Delta(x)$. We only need the fact that the determinant
is related to the characteristic function $\varphi$ by
\begin{equation*}
  \varphi(x) = \varphi(e)\Delta(x)^{-n/r},
\end{equation*}
where $r$ denotes the rank of the cone.
If $x=ge$ we have
\begin{equation} \label{eq:determinantrelation}
  \Delta(x) = \Delta(ge) = |\mathrm{Det}(g)|^{r/n}.
\end{equation}

The following growth estimates hold for functions in $\mathcal{S}_\cone$
(see Lemma 3.11 in \cite{Bekolle2004}):
\begin{lemma}
  \label{lem:rapidoncone}
  If $\phi\in \mathcal{S}_\cone$ and $k,l$ are non-negative 
  integers, then there is an $N=N(k,l)$ and a constant $C_{N}$ such that
  \begin{equation*}
    |\widehat{\phi}(w)| 
    \leq C_{N}\| \widehat{\phi}\|_{N} \frac{\Delta(w)^{l}}{(1+|w|)^k}.
  \end{equation*}
\end{lemma}

\subsection{Besov spaces on symmetric cones}
The cone $\cone$ can be identified as a Riemannian manifold
$\cone = \grG_0(\cone)/\grK$ where $\grK$ is the compact
group fixing $e$. The Riemannian metric
in this case is defined by
\begin{equation*}
  \dup{u}{v}_y = \ip{g^{-1}u}{g^{-1}v}
\end{equation*}
for $u,v$ tangent vectors to $\cone$ at $y=ge$. Denote the balls
of radius $\delta$ centered at $x$ by $B_\delta(x)$. For $\delta>0$ and 
$R\geq 2$ the points $\{ x_j\}$ are called a $(\delta,R)$-lattice if
\begin{enumerate}
\item $\{ B_\delta(x_j)\}$ are disjoint, and
\item $\{ B_{R\delta}(x_j)\}$ cover $\cone$.
\end{enumerate}

We now fix a $(\delta,R)$-lattice $\{ x_j\}$ with $\delta=1/2$ and $R=2$.
Then there are functions $\psi_j\in\mathcal{S}_\cone$,
such that $0\leq\widehat{\psi}_j\leq 1$, $\mathrm{supp}(\widehat{\psi}_j) \subseteq B_2(x_j)$,
$\widehat{\psi}_j$ is one on $B_{1/2}(x_j)$ and $\sum_j \widehat{\psi}_j =1$ on $\cone$.
Using this decomposition of the cone, the Besov space norm for $1\leq p,q<\infty$ and $s\in \mathbb{R}$ is
defined in \cite{Bekolle2004} by
\begin{equation*}
  \| f\|_{\besov^{p,q}_s} = \left( \sum_j \Delta(x_j)^{-s}\| f*\psi_j\|_p^q \right)^{1/q}.
\end{equation*}
The Besov space $\besov^{p,q}_s$ consists of 
the equivalence classes of tempered distributions $f$ in
$\mathcal{S}_\cone' \simeq 
\{ f\in \mathcal{S}'(V) \mid \supp(\widehat{f})\subseteq \overline{\cone}\}
/\mathcal{S}'_{\partial \cone}$ 
for which $\| f\|_{\besov^{p,q}_s} <\infty$.
\begin{theorem}
  \label{thm:normequiv}
  Let $\psi$ be a function in $\mathcal{S}_\cone$ for which 
  $1_{B_{1/2}(e)} \leq \widehat{\psi}\leq 1_{B_2(e)}$.
  Defining $\psi_h$ by
  \begin{equation*}
    \widehat{\psi}_h(w) = \widehat{\psi}(h^{-1}w),
  \end{equation*}
  then
  \begin{equation*}
    \| f\|_{\besov^{p,q}_s} \simeq \left( \int_\grH \| f*\psi_h \|_p^q \mathrm{Det}(h)^{-sr/n}
      \,dh \right)^{1/q}
  \end{equation*}
  for $f\in \mathcal{S}_\cone^*$.
\end{theorem}

\begin{proof}

Before we prove the theorem, let us note that
\begin{equation*}
  (\psi_h)_g = \psi_{gh}.
\end{equation*}

The cover of $\cone$ corresponds to a cover
of $\grH$: if $h_j\in \grH$ is 
such that $x_j = h_je$ then $h_jU$ covers $\grH$ with
$U = \{ h\in \grH \mid he \in B_1(e)\}$. 
\begin{align*}
  \left( \int_\grH \| f*\psi_h \|_p^q \det(h)^{-sr/n} \,dh \right)^{1/q}
  &\leq \left( \sum_j \int_{h_jU} \| f*\psi_h \|_p^q \det(h)^{-sr/n} \,dh \right)^{1/q} \\
  &\leq C \left( \sum_j \int_{h_jU} \| f*\psi_h \|_p^q \det(h_j)^{-sr/n}
    \,dh \right)^{1/q}.
\end{align*}
In the last inequality we have used that, if $h\in h_jU$ then
$\det(h) \sim \det(h_j)$ uniformly in $j$. This follows since for 
$h\in h_jU$, $\det(h) = \det(h_j)\det(u)$ for some $u\in U$,
and since $U$ is bounded (compact) there is a $\gamma$ such that
$1/\gamma \leq \det(u) \leq \gamma$ uniformly in $j$.
For $h\in h_jU$ all the functions $\widehat{\psi}_{h_j^{-1}h}$
have compact support contained in a larger compact set.
Therefore there is an finite set $I$ such that
$$
\widehat{\psi}_{h_j^{-1}h} 
= \widehat{\psi}_{h_j^{-1}h} \sum_{i\in I} \widehat{\psi}_i.
$$
Then, for $h\in h_jU$ we get
$$
\widehat{\psi}_{h} 
= \widehat{\psi}_{h} \sum_{i\in I} (\widehat{\psi}_i)_{h_j},
$$
and, since $\psi_h\in L^1(V)$,
\begin{equation*}
  \| f*\psi_h \|_p 
  \leq \sum_{i\in I} \| f*(\psi_i)_{h_j} \|_p.
\end{equation*}
So we get
\begin{align*}
  \left( \int_\grH \| f*\psi_h \|_p^q \det(h)^{-sr/n} \,dh \right)^{1/q}
  &\leq C \sum_{i\in I}
  \left( \sum_j \| f*(\psi_i)_{h_j} \|_p^q \det(h_j)^{-sr/n} 
  \right)^{1/q} \\
  &\leq C \| f\|_{\besov_s^{p,q}}.
\end{align*}
In the last inequality 
we used that each of the collections $\{ (\psi_i)_{h_j}\}_j$ partitions the
frequency plane, and the expression can thus be estimated
by a Besov norm (see Lemma 3.8 in \cite{Bekolle2004}).

The opposite inequality can be obtained in a similar fashion.

\end{proof}

\section{A Wavelet Characterization of 
  Besov spaces on symmetric cones}
We will now show that the Besov spaces can be
characterized as coorbits for the group
$\grG=\grH\ltimes V$, with isomorphism given by
the mapping $f\mapsto \widetilde{f}$ from
$\mathcal{S}'_\Omega$ to $\mathcal{S}_\Omega^*$ defined via
$\dup{\widetilde{f}}{\phi} = f(\overline{\phi})$.
Notice that convolution $f*\phi$ can be expressed 
via the conjugate linear dual pairing as
\begin{equation*}
  f*\phi(x) = \dup{\widetilde{f}}{\tau_x\phi^*}.
\end{equation*}

\subsection{Wavelets and coorbits on symmetric cones}

The group of interest to us is the semidirect product
$\grG=\grH\ltimes V$ with group composition
\begin{equation*}
  (h,x)(h_1,x_1) = (hh_1,hx_1+x).
\end{equation*}
Here $\grH=\grA\grN$ is the connected solvable subgroup of the connected
component of the automorphism group on $\cone\subseteq V$.
If $dh$ denotes the left Haar measure on $\grH$ and $dx$ the Lebesgue
measure on $V$, then the left Haar measure on $\grG$ is given
by $\frac{dx\,dh}{\det(h)}.$

The quasi regular representation of this group on $L^2(V)$ is
given by
\begin{equation*}
  \pi(h,x) f(t) = \frac{1}{\sqrt{\mathrm{Det}(h)}} f(h^{-1}(t-x)),
\end{equation*}
and it is irreducible and square integrable on 
$L^2_\cone = \{f\in L^2(V)\mid \supp(\widehat{f})\subseteq\cone  \}$
(see \cite{Fuhr1996,Fabec2003}).
In frequency domain the representation becomes
\begin{equation*}
  \widehat{\pi}(h,x) \widehat{f}(w) = \sqrt{\det(h)}e^{-i\ip{x}{w}}\widehat{f}(h^*w).
\end{equation*}
By $\pi$ we will also denote the restriction of $\pi$ to $\mathcal{S}_\cone$.
\begin{remark}
  The norm equivalence we have shown in Theorem~\ref{thm:normequiv}
  is related to the unitary representation
  \begin{equation*}
    \rho(h,x) f(t) = {\sqrt{\mathrm{Det}(h^*)}} f(h^{*}(t-x))
  \end{equation*}
  and not the representation $\pi$. 
  However, Lemma~\ref{lem:newhaar} 
  allows us to make a change of variable $h\mapsto (h^{*})^{-1}$
  in order to relate the norm equivalence to $\pi$.
\end{remark}

\begin{lemma}
  \label{lemma:repn}
  The representation $\pi$ of $\grG$ on $\mathcal{S}_\cone$
  is continuous, and if $\psi$ is the function
  from from Theorem~\ref{thm:normequiv}, then the function $\phi=\psi^*$
  is a cyclic vector for $\pi$.
\end{lemma}

\begin{proof}
  The Fourier transform ensures that this is equivalent to 
  showing that $\widehat{\pi}$ is a continuous representation.
  The determinant is continuous, so we will investigate
  the $L^\infty$-normalized representation instead.
  For $f\in \mathcal{S}(V)$ with support
  in $\cone$ define 
  $$f_{h,x}(w) = f(h^*w) e^{-i\ip{x}{w}},$$ 
  for $h\in \grH$ and $x\in V$.
  Since $h^*w \in \cone$ if $w\in  \cone$ we see that
  $f_{h,x}$ is a Schwartz function 
  supported in $\cone$, so $\mathcal{S}_\cone$ is
  $\pi$-invariant. 

  We now check that $f_{h,x} \to f$
  in the Schwartz semi-norms as $h\to I$ and $x\to 0$. 
  Taking one partial derivative we see that
  \begin{align*}
    \frac{\partial f_{h,x}}{\partial w_k} (w)
    &- \frac{\partial f}{\partial w_k}(w) \\
    &= \sum_l h_{lk}\frac{\partial f}{\partial w_l} (h^*w)e^{-i\ip{x}{w}} 
    - iw_k f(h^*w) e^{-i\ip{x}{w}} - \frac{\partial f}{\partial w_k}(w) \\
    &= 
    (h_{kk}-1) \frac{\partial f}{\partial w_k} (h^*w)e^{-i\ip{x}{w}} +
    \sum_{l\neq k} h_{lk}\frac{\partial f}{\partial w_l} (h^*w)e^{-i\ip{x}{w}} 
    \\ &\qquad 
    - iw_k f(h^*w) e^{-i\ip{x}{w}}
    + \frac{\partial f}{\partial w_k}(h^*w)e^{-i\ip{x}{w}} -
    \frac{\partial f}{\partial w_k}(w) \\
    &= \sum_{|\beta|\leq |\alpha|} 
    c_\beta(h,x) \pdif^{\beta}f(h^*w)e^{-i\ip{x}{w}} + 
    (\pdif^\alpha f(h^*w)e^{-i\ip{x}{w}} -\pdif^\alpha f(w)),
  \end{align*}
  where $\alpha = e_k$ and $c_\beta(h,x)\to 0$ as $(h,x)\to (I,0)$.
  By repeating the argument we get
  \begin{align*}
    \pdif^\alpha f_{h,x}(w) &- \pdif^\alpha f(w) \\
    &= \sum_{|\beta|\leq |\alpha|} 
    c_\beta(h,x) \pdif^{\beta}f(h^*w)e^{-i\ip{x}{w}} + 
    (\pdif^\alpha f(h^*w)e^{-i\ip{x}{w}} -\pdif^\alpha f(w)).
  \end{align*}
  where $c_\beta(h,x)\to 0$ as $(h,x)\to (I,0)$.
  Using the fact that $|w| = |(h^*)^{-1}h^*w| \leq \| (h^*)^{-1}\| |h^*w|$ we see
  that $\frac{(1+|w|)^N}{(1+|h^*w|)^N} \leq C_N(h)$, where $C_N(h)$ 
  depends continuously on $h$.
  For $|\alpha|\leq N$ we thus get
  \begin{align*}
    (1+|w|)^N &|\pdif^\alpha f_{h,x}(w) - \pdif^\alpha f(w)| \\
    &\leq C_N(h) \sum_{|\beta|\leq |\alpha|} 
    c_\beta(h,x) (1+|h^*w|)^N |\pdif^{\beta}f(h^*w)|\\
    &\qquad + 
    (1+|w|)^N |\pdif^\alpha f(h^*w)e^{-i\ip{x}{w}} -\pdif^\alpha f(w)| \\
    &\leq 
    C_N(h) \sum_{|\beta|\leq |\alpha|} c_\beta(h,x) \| f\|_N + 
    (1+|w|)^N |\pdif^\alpha f(h^*w)e^{-i\ip{x}{w}} -\pdif^\alpha f(w)|.
  \end{align*}
  Since $c_\beta (h,x)$ tend to $0$ as $(h,x)\to (I,0)$,
  we investigate the remaining term 
  \begin{equation*}
   |\pdif^\alpha f(h^*w)e^{-i\ip{x}{w}} -\pdif^\alpha f(w)|
   \leq |\pdif^\alpha f(h^*w)-\pdif^\alpha f(w)| + |\pdif^\alpha f(w)(e^{-i\ip{x}{w}}-1)|
  \end{equation*}
  First, let $\gamma(t)=w+t(h^*w-w)$. For $|\alpha|=N$ we get
  \begin{align*}
    |\pdif^\alpha f(h^*w) &- \pdif^\alpha f(w)|(1+|w|^2)^{N} \\
    & \leq \int_0^1
    |\nabla{\pdif^\alpha f}(\gamma_{h,w}(t))||\gamma_{h,w}'(t)| (1+|w|^2)^{N} \,dt\\
    &\leq 
    \|h^*-I\| \int_0^1 
    |\nabla{\pdif^\alpha f}(\gamma_{h,w}(t))|(1+|\gamma(t)|^2)^{N+1} 
    \left( \frac{1+|w|^2}{1+|\gamma(t)|^2}\right)^{N+1}  \,dt\\
    &\leq C\|h^*-I\| \|f\|_{N+1},
  \end{align*}
  where the constant $C$ is uniformly bounded in $h$.
  Next let $\gamma(t)=tx$, then
  \begin{align*}
    (1+|w|)^N |\pdif^\alpha f(w)(e^{-i\ip{x}{w}}-1)|    
    &\leq (1+|w|)^N |\pdif^\alpha f(w)| |\int_0^1 -iw\gamma'(t) e^{-it\ip{x}{w}}) \,dt| \\
    &\leq (1+|w|)^N |\pdif^\alpha f(w)| |w||x| \\
    &\leq \|f\|_{N+1} |x|.
  \end{align*}
  This shows that the representation $\pi$ is continuous on 
  $\mathcal{S}_\cone$.
  
  To show cyclicity, assume that $\widetilde{f}$ is in 
  $\mathcal{S}^*_\Omega$ and 
  $\dup{\widetilde{f}}{\pi(a,x)\phi}=0$. Notice that
  $\dup{\widetilde{f}}{\pi(a,x)\phi} = f*\psi_{(h^{-1})^*}(x)$,
  where $f$ is the tempered distribution in $\mathcal{S}'_\Omega$
  corresponding to $\widetilde{f}$.
  By the norm equivalence of Theorem~\ref{thm:normequiv} and 
  Lemma~\ref{lem:newhaar}, we see that 
  $f=0$ in all Besov spaces $\besov_s^{p,q}$ and thus also 
  in $\mathcal{S}_\cone'$ 
  (see \cite{Bekolle2004} Lemma 3.11 and 3.22 and note
  that $\mathcal{S}_\cone'$ is equipped with the weak$^*$ topology).
  This proves that $\widetilde{f}=0$ and $\phi$ is cyclic.
\end{proof}

For $\psi\in\mathcal{S}_\cone$ define the wavelet transform of 
$f\in\mathcal{S}_\cone^*$ by
\begin{equation*}
  W_\psi(f)(h,x) = \dup{f}{\pi(h,x)\psi}.
\end{equation*}
Under certain assumptions on $\psi$ we get a reproducing formula.
\begin{lemma}
  \label{lemma:reprod}
  If $\psi\in\mathcal{S}_\cone$ is such that 
  $\widehat{\psi}$ has compact support and 
  \begin{equation*}
    \int_\grH |\widehat{\psi}(h^*e)|^2 \,dh =1,
  \end{equation*}
  then 
  $$W_\psi(f)*W_\psi(\psi) = W_\psi(f)$$
  for all $f \in \mathcal{S}_\cone^*$.
  Here the convolution is the group convolution on $\grG=\grH\ltimes V$.
\end{lemma}
\begin{proof}
  For $\phi$ we denote by
  $\phi^h$ the function defined by
  \begin{equation*}
    \widehat{\phi^h}(w) = \widehat{\phi}(h^*w).
  \end{equation*}
  Then
  $\phi_1^h*\phi_2^h = |\det(h)|(\phi_1*\phi_2)^h$, and
  \begin{align*}
    W_\psi(f)*W_\psi(\psi) (h_1,x_1)
    &= \frac{1}{\sqrt{|\det(h_1)|}} \int_\grH 
    \dup{f}{\tau_{x_1}\psi^{h_1}*(\psi^*)^{h}*\psi^h} \,\frac{dh}{|\det(h)|^2} \\
    &= \frac{1}{\sqrt{|\det(h_1)|}} \int_\grH 
    \dup{f}{\tau_{x_1}\psi^{h_1}*(\psi^**\psi)^h} \,\frac{dh}{|\det(h)|}.
  \end{align*}
  
  The function inside the last integral is continuous,
  so it is enough to show that for $\phi\in\mathcal{S}_\cone$ the net
  \begin{equation*}
    g_C(x) = \int_C \phi*(\psi^**\psi)^h \,\frac{dh}{|\det(h)|},
  \end{equation*}
  converges to $\phi$ in $\mathcal{S}_\cone$ for growing compact sets 
  $C\to \grH$.
  By the assumption on $\psi$ we get that 
  $\widehat{g}_C\to \widehat{\phi}$ pointwise. Thus 
  we only need to show that $g_C$ converges, which will happen if
  the integral 
  $$
  \int_\grH \sup_x (1+|x|^2)^N|\pdif^\alpha \phi*(\psi^**\psi)^h(x)| \,\frac{dh}{|\det(h)|} 
  <\infty
  $$
  is finite for all $N$ and $\alpha$. Since both $\pdif^\alpha \phi$ and 
  $\psi^**\psi$ are in $\mathcal{S}_\cone$, we need only focus on showing that
  \begin{equation*}
    \int_\grH \sup_x (1+|x|^2)^N|\phi_1*\phi_2^h(x)| \,\frac{dh}{|\det(h)|} 
    <\infty
  \end{equation*}
  for all $N$ and any $\phi_1,\phi_2\in\mathcal{S}_\cone$.
  We can further assume that $\widehat{\phi}_2$ has compact support.
  Note that the Parseval identity, integration by parts and
  the fact that $\widehat{\phi}_1,\widehat{\phi}_2$ vanish on the
  boundary of $\cone$, give
  \begin{align*}
    |\phi_1*\phi_2^h(x) |
    &= \left|\int_\cone \widehat{\phi}_1(w)\widehat{\phi}_2(h^*w)e^{i\ip{x}{w}}\,dw \right|\\
    &\leq \frac{1}{|x^\alpha|}
    \int_\cone \sum_{|\beta|\leq |\alpha|} |p_\beta(h^*)|
    |\pdif^\beta\widehat{\phi}_1(w) \pdif^{\alpha-\beta}\widehat{\phi}_2(h^*w)|\,dw.
  \end{align*}
  Here $p_\beta(h^*)$ is a polynomium in the entries of $h^*$.
  Choosing $|\alpha|$ large enough takes care of the terms
  $(1+|x|^2)^N$ for large $|x|$ (and for small $|x|$ we use $\alpha=0$), so
  \begin{align*}
    \sup_x |\phi_1*\phi_2^h(x)|(1+|x|^2)^N
    &\leq \sum_{|\beta|\leq |\alpha|} |p_\beta(h^*)| \int_\cone
    |\pdif^\beta\widehat{\phi}_1(w) \pdif^{\alpha-\beta}\widehat{\phi}_2(h^*w)|\,dw.
  \end{align*}
  Each partial derivative is again in $\mathcal{S}(V)$ and with support in 
  $\cone$, so we investigate terms of the general form 
  $\int |\widehat{\phi}_1(w) \widehat{\phi}_2(h^*w)|\,dw$.
  Denote by $h_w$ the unique element
  in $\grH$ for which $w=h_we$, then
  \begin{align*}
    \int_\grH \int_\cone p(\| h\|) |\widehat{\phi}_1(w)||\widehat{\phi}_2(h^*w)|\,dw\,dh
    &= \int_\grH \int_\cone p(\| (h_w^*)^{-1}h\|) |\widehat{\phi}_1(w)||\widehat{\phi}_2(h^*e)|\,dw\,dh. \\
    \intertext{
      Now $\widehat{\phi}_2$ 
      is assumed to have compact support and thus the integral 
      over $\grH$ is finite, so we get}
   &\leq C \int_\cone p(1/\| h_w\|) |\widehat{\phi}_1(w)|\,dw. \\
 \end{align*}
 This will be finite, because the estimate
 $|\widehat{\phi}_1(w)|\leq C \frac{\Delta(w)^l}{(1+|w|^2)^k}$.
 When $\|h_w\|\sim |w|$ is close to zero we use $l$ sufficiently large
 and for large $\|h_w\|\sim|w|$ the integral is finite for $k$ large enough.
 This finishes the proof.
\end{proof}

For $1\leq p,q<\infty$ and $s\in\mathbb{R}$ 
define the mixed norm Banach space $L^{p,q}_s(\grG)$ on the group $\grG$ to be 
the measurable functions for which
\begin{equation*}
  \|F \|_{L^{p,q}_s} 
  := \left( \int_\grH \left( \int_V |F(h,x)|^p \,dx \right)^{q/p} |Det(h)|^s \,dh
  \right)^{1/q} <\infty.
\end{equation*}

\begin{lemma}
  \label{lemma:SinBesov}
  For $\psi,\phi\in \mathcal{S}_\cone$ the wavelet transform $W_\psi(\phi)$ is in
  $L^{p,q}_s(\grG)$ for $1\leq p,q<\infty$ and any real $s$. 
\end{lemma}

\begin{proof}
  Since
  $
  W_\psi(\phi)(h,x) =
  \phi*\psi^*_{(h^{-1})^*}(x),
  $
  this follows from the norm equivalence of Theorem~\ref{thm:normequiv}
  coupled with Lemma~\ref{lem:newhaar}, 
  as well as the fact that a function
  in $\mathcal{S}_\cone$ is in any Besov space
  (see Proposition 3.9 in \cite{Bekolle2004}).
\end{proof}

This verifies that the requirements of Theorem~\ref{thm:coorbitsduality}
are satisfied. It also shows that the representation involved has integrable
matrix coefficients, which is the basis for the investigation in 
\cite{Feichtinger1989a}.
We thus complete our wavelet characterization of the 
Besov spaces by the following result. Remember
that $\widetilde{f}\in\mathcal{S}^*_\Omega$ corresponds to
$f\in\mathcal{S}'_\Omega$ via $\dup{\widetilde{f}}{\phi} = f(\overline{\phi})$.

\begin{theorem}
  Given $1\leq p,q<\infty$ and $s\in\mathbb{R}$ let $s'=sr/n-q/2$. 
  If $\phi$ is the cyclic vector from Lemma~\ref{lemma:repn}
  normalized to also satisfy Lemma~\ref{lemma:reprod},
  then the mapping $f\mapsto \widetilde{f}$ (restricted to $\besov_s^{p,q}$)
  is a Banach space isomorphism from the Besov space $\besov_s^{p,q}$ 
  to the coorbit $\Co_{\mathcal{S}_\cone}^\phi L^{p,q}_{s'}(\grG)$ 
  for the representation $\pi$.
\end{theorem}

\begin{proof}
  We will use Theorem~\ref{thm:normequiv} 
  to determine $s'$.
  Let $\phi=c\psi^*$, and notice that 
  \begin{equation*}
    \dup{\widetilde{f}}{\pi(h,x)\phi}
    = c \sqrt{|\det(h)|} f*\psi_{(h^*)^{-1}}(x).
  \end{equation*}
  Then by Lemma~\ref{lem:newhaar} and 
  Theorem~\ref{thm:normequiv} we get that
  \begin{equation*}
    \|W_\phi(\widetilde{f})\|_{L^{p,q}_{s'}}
    = C \left( \int_\grH \| f*\psi_h \|_p^q \mathrm{Det}(h)^{-q/2-s'}
      \,dh \right)^{1/q},
  \end{equation*}
  which is equivalent to $\| f\|_{\besov^{p,q}_s}$ if
  $-q/2-s' = -sr/n$.
\end{proof}

\subsection{Atomic decompositions}
\label{sec:atomdecompbesov}

In order to obtain atomic decompositions and frames
from Theorems~\ref{thm:discretization} and \ref{thm:12},
we need to show that $\mathcal{S}_\cone$ are 
smooth vectors for $\pi$. A vector $\psi\in \mathcal{S}_\cone$ 
is called smooth if $g\mapsto \pi(g)\psi$ is smooth $\grG\to \mathcal{S}_\cone$.
For smooth vectors $\psi$ define a representation of $\mathfrak{g}$ by
\begin{equation*}
  \pi^{\infty}(X)\psi = \frac{d}{dt}\Big|_{t=0} \pi(exp(tX))\psi.
\end{equation*}
This also induces a representation of the universal
enveloping algebra $\mathcal{U}(\mathfrak{g})$ which
we also denote $\pi^\infty$.
\begin{theorem}
  The space $\mathcal{S}_\cone$ is the space of smooth vectors for the
  representation $(\pi,\mathcal{S}_\cone)$, and
  $(\pi^\infty,\mathcal{S}_\cone)$ is
  a representation of both $\mathfrak{g}$ 
  and $\mathcal{U}(\mathfrak{g})$.
\end{theorem}

\begin{proof}
  Again, the determinant does not change the smoothness
  of vectors so we work with the $L^\infty$-normalized representation.
  Let $\gamma(t)=(h(t),x(t))$ be a smooth curve in $\grG$ with $\gamma(0)=(I,0)$
  and $\gamma'(t) = (H,X)$, then 
  the pointwise derivative (on the frequency side)
  of functions $f_{h,x}(w) = f(h^*w)e^{-i\ip{x}{w}}$ is
  \begin{equation*}
    \frac{d}{dt}\Big|_{t=0} f(h(t)^*w)e^{-i\ip{x(t)}{w}}
    = (H^*w) \cdot \nabla f(w) - iX \cdot w f(w).
  \end{equation*}
  This is another Schwartz function
  and we will show it is also the limit of the derivative
  in $S_\cone$.
  \begin{align*}
    \frac{1}{t}(f_{h(t),x(t)}(w) &-f (w))
    - (H^*w) \cdot \nabla f(w) + iX \cdot w f(w), \\
    &= \frac{1}{t} \int_0^t ((h'(s)^*w)\cdot\nabla f(h(s)^*w)e^{-i\ip{x(s)}{w}} 
    - (H^*w) \cdot \nabla f(w) \,ds \\
    &\quad+ \frac{1}{t} \int_0^T iX \cdot w f(w) - ix'(t)\cdot w f(h(s)^*w)\,ds
  \end{align*}
  From the proof of Lemma~\ref{lemma:repn} it is evident that
  each term inside the integral approaches $0$ in the Schwartz topology,
  and the proof is complete.
\end{proof}

This result proves that any vector in $\mathcal{S}_\cone$ is both $\pi$-weakly
and $\pi^*$-weakly differentiable of all orders, and this ensures
that $W_{\pi(D^\alpha)\psi}(\psi)$ and $W_\psi({\pi(D^\alpha)\psi})$ are in 
$L^1_s(\grG)$ for all $s$ by Lemma~\ref{lemma:SinBesov}. Thus the continuities 
required by Theorems \ref{thm:discretization} and \ref{thm:12}
are satisfied and we conclude with the promised atomic decompositions.
\begin{corollary}
  Let $s\in \mathbb{R}$ and $1\leq p,q<\infty$ be given.
  There exists an index set $I$, and a well-spread sequence of points
  $\{ (h_i,x_i)\}_I \subseteq \grG$, such that 
  the collection $\pi(h_i,x_i) \psi$ forms both a Banach frame 
  and an atomic decomposition for $\besov^{p,q}_s$ with sequence space 
  $(L^{{p,q}}_{s'}(\grG))^\#$ when $s'=sr/n-q/2$.
\end{corollary}

\bibliographystyle{plain}


\begin{thebibliography}{10}

\bibitem{Bekolle2000}
D.~B{\'{e}}koll{\'{e}}, A.~Bonami, and G.~Garrig{\'{o}}s.
\newblock Littlewood-{P}aley decompositions related to symmetric cones.
\newblock {\em IMHOTEP J. Afr. Math. Pures Appl.}, 3(1):11--41, 2000.

\bibitem{Bekolle2004}
D.~B{\'{e}}koll{\'{e}}, A.~Bonami, G.~Garrig{\'{o}}s, and F.~Ricci.
\newblock Littlewood-{P}aley decompositions related to symmetric cones and
  {B}ergman projections in tube domains.
\newblock {\em Proc. London Math. Soc. (3)}, 89(2):317--360, 2004.

\bibitem{Bekolle2001}
D.~B{\'{e}}koll{\'{e}}, A.~Bonami, M.~M. Peloso, and F.~Ricci.
\newblock Boundedness of {B}ergman projections on tube domains over light
  cones.
\newblock {\em Math. Z.}, 237(1):31--59, 2001.

\bibitem{Christensen2012}
J.~G. Christensen.
\newblock Sampling in reproducing kernel banach spaces on lie groups.
\newblock {\em Journal of Approximation Theory}, 164(1):179 -- 203, 2012.

\bibitem{Christensen2009}
J.~G. Christensen and G.~{\'{O}}lafsson.
\newblock Examples of coorbit spaces for dual pairs.
\newblock {\em Acta Appl. Math.}, 107(1-3):25--48, 2009.

\bibitem{Christensen2011}
J.~G. Christensen and G.~{\'{O}}lafsson.
\newblock Coorbit spaces for dual pairs.
\newblock {\em Appl. Comput. Harmon. Anal.}, 31(2):303--324, 2011.

\bibitem{Fabec2003}
R.~Fabec and G.~{\'{O}}lafsson.
\newblock The continuous wavelet transform and symmetric spaces.
\newblock {\em Acta Appl. Math.}, 77(1):41--69, 2003.

\bibitem{Faraut1994}
Jacques Faraut and Adam Kor{\'a}nyi.
\newblock {\em Analysis on symmetric cones}.
\newblock Oxford Mathematical Monographs. The Clarendon Press Oxford University
  Press, New York, 1994.
\newblock Oxford Science Publications.

\bibitem{Feichtinger1985}
H.~G. Feichtinger and P.~Gr{\"{o}}bner.
\newblock Banach spaces of distributions defined by decomposition methods. {I}.
\newblock {\em Math. Nachr.}, 123:97--120, 1985.

\bibitem{Feichtinger1989a}
H.~G. Feichtinger and K.~Gr{\"{o}}chenig.
\newblock {B}anach spaces related to integrable group representations and their
  atomic decompositions. {I}.
\newblock {\em J. Funct. Anal.}, 86(2):307--340, 1989.

\bibitem{Fuhr1996}
H.~F{\"{u}}hr.
\newblock Wavelet frames and admissibility in higher dimensions.
\newblock {\em J. Math. Phys.}, 37(12):6353--6366, 1996.

\bibitem{Fuhr2013}
    H.~F{\"u}hr.
   \newblock Coorbit spaces and wavelet coefficient decay over general dilation groups.
   \newblock Available under http://arxiv.org/abs/1208.2196v3

\bibitem{Grochenig1991}
K.~Gr{\"{o}}chenig.
\newblock Describing functions: atomic decompositions versus frames.
\newblock {\em Monatsh. Math.}, 112(1):1--42, 1991.

\bibitem{Peetre76}
J.~Peetre.
\newblock {\em New thoughts on {B}esov spaces}.
\newblock Mathematics Department, Duke University, Durham, N.C., 1976.
\newblock Duke University Mathematics Series, No. 1.

\bibitem{Triebel1988b}
H.~Triebel.
\newblock Characterizations of {B}esov-{H}ardy-{S}obolev spaces: a unified
  approach.
\newblock {\em J. Approx. Theory}, 52(2):162--203, 1988.

\end{thebibliography}
\end{document}